\documentclass[a4paper, 12pt]{article}
\usepackage{amssymb}
\usepackage{amsmath}
\usepackage{amsthm}
\usepackage[dvipdfmx]{graphicx}
\usepackage{verbatim,enumerate}
\usepackage[active]{srcltx}
\usepackage{cite}
 \newtheorem{theorem}{Theorem}[section]
 \newtheorem*{theorem*}{Theorem}
 \newtheorem*{lemma*}{Lemma}
 \newtheorem{proposition}[theorem]{Proposition}
 \newtheorem{fact}[theorem]{Fact}
 \newtheorem{fact*}{Fact}
 \newtheorem{lemma}[theorem]{Lemma}
 
\theoremstyle{definition}

 \newtheorem*{remark*}{Remark}
 
 \newtheorem{example}[theorem]{Example}
\numberwithin{equation}{section}

\usepackage[usenames]{color}

\newcommand{\R}{\boldsymbol{R}}

\newcommand{\trans}[1]{{\vphantom{#1}}^t{\!#1}}
\renewcommand{\phi}{\varphi}

\newcommand{\A}{\mathcal{A}}

\newcommand{\pmt}[1]{{\begin{pmatrix} #1  \end{pmatrix}}}

\newcommand{\mycomment}[1]{}
\makeatletter
  \newcommand{\subsubsubsection}{\@startsection{paragraph}{4}{\z@}%
    {-1ex \@plus -1ex \@minus -.2ex}%
    {1.0ex \@plus.2ex}
    {\reset@font\bfseries\normalsize}
  }
  \makeatother
\begin{document}
\begin{center}
{\large {\bf 
Singular surfaces 
of revolution with prescribed unbounded mean curvature}}
\\[2mm]
December 20, 2017
\\[2mm]
\renewcommand{\thefootnote}{\fnsymbol{footnote}}
Luciana F. Martins,
Kentaro Saji,
Samuel P. dos Santos
 and
Keisuke Teramoto
\footnote[0]{
Saji is partly supported by the
JSPS KAKENHI Grant Number 26400087.
Santos was supported by
CAPES and S\~ao Paulo Research Foundation-FAPESP, grant 2016/21226-5.
Teramoto is partly supported by the
JSPS KAKENHI Grant Number 17J02151.}

\begin{abstract}
{\small 
We give an explicit formula for singular surfaces 
of revolution with prescribed unbounded mean curvature.
Using it, we give conditions for singularities 
of that surfaces. Periodicity of that surface
is also discussed.}
\end{abstract}
\end{center}
\section*{Introduction}
In this note, we study 
singular surfaces of revolution.
Let $I\subset \R$ be a subset, and
$\gamma:I\to\R^2$ a $C^\infty$ plane curve.
We set $\gamma(t)=(x(t),y(t))$ $(y>0)$,
and set the revolution surface 
\begin{equation}\label{eq:revol}
s(t,\theta)=\big(x(t),y(t)\cos\theta,y(t)\sin\theta\big)
\end{equation}
of $\gamma$.
The curve $\gamma$ is called the {\it profile curve\/}
of $s$.
We denote by $H(t)$ the mean curvature of $s(t,\theta)$.
Given a $C^\infty$ function $H(t)$ on $I$, 
it is given by Kenmotsu \cite{k1}
that the concrete solution of the profile curve 
$(x(t),y(t))$ satisfying
the revolution surface $s(t,\theta)$ has the mean curvature
$H(t)$.
Moreover, the periodicity of $s$ is also studied \cite{k2}.

On the other hand, in the recent decades, 
there are several articles 
concerning differential geometry of singular curves and
surfaces, namely,
curves and surfaces with singular points,
in the $2$ and $3$ dimensional Euclidean spaces 
\cite{bw,fh,ft,hnuy,hnuy2,irrf,MSUY,nuy,OT,front,ss}.
If the profile curve $\gamma$ is regular, then
the mean curvature $H$ is differentiable on $I$, but
if $\gamma$ has a singularity, then
$H$ may diverge \cite{front} (see also \cite{MSUY}).
Given a $C^\infty$ function $H$ defined on $I\setminus P$,
where $P$ is a discrete set,
we give a concrete solution $\gamma=(x,y)$
such that the mean curvature of revolution surface
of $\gamma$ is $H$.
Moreover, we give conditions for the 
fundamental singularities of $\gamma$.
We also discuss the periodicity of the surface.

\section{Construction of singular surfaces of revolution}
Let $I\subset \R$ and $\gamma:I\to\R^2$ be a
$C^\infty$ curve.
We set $\gamma(t)=(x(t),y(t))$, and assume 
$y(t)>0$ for any $t\in I$.
We assume that there exists a $C^\infty$ map $\phi:I\to\R$
satisfying that
$\gamma'(t)$ and $(\cos\phi(t),\sin\phi(t))$ are linearly dependent
for any $t\in I$.
Then we have a function $l:I\to\R$ such
that
$$
\gamma'(t)=l(t)e(t),\quad e(t)=(\cos\phi(t),\sin\phi(t)).
$$
This condition is equivalent to saying that $\gamma$ is a frontal
(see Section \ref{sec:sing} for detail).
We choose the unit normal vector of 
the revolution surface $s$
by 
\begin{equation}\label{eq:nuconv}
\nu(t,\theta)=\big(\sin\phi(t),
-\cos\phi(t)\cos\theta,-\cos\phi(t)\sin\theta\big).
\end{equation}
Then the mean curvature $H$ can be given
on the regular set of $s$.
We have the following lemma.
\begin{lemma}\label{lem:hetabdd}
The function\/
$Hl$ can be extended to a\/ $C^\infty$ function on\/ $I$.
\end{lemma}
\begin{proof}
By a direct computation, $Hl$ with respect to the 
unit normal vector \eqref{eq:nuconv} can be calculated by
$$
H(t)l(t)
=
\dfrac{1}{2}\left(
\dfrac{\cos\phi(t)}{y(t)} - \dfrac{\phi'(t)}{l(t)}
\right)l(t),
$$
where ${}'=d/dt$.
Since $y>0$, this proves the assertion.
\end{proof}
See \cite[Proposition 3.8]{MSUY} for more detailed behavior
of the mean curvature for the case of cuspidal edges.
We remark that the case $y=0$ is already considered
in \cite{k1}.

Conversely, given a $C^\infty$ function $H:I\setminus P\to \R$,
where $P$ is a discrete set,
and function $l:I\to \R$ satisfying 
that $Hl$ is a $C^\infty$ function
on $I$ and $l^{-1}(0)=P$,
we look for a surface of revolution with the profile curve
$\gamma$ whose mean curvature 
with respect to \eqref{eq:nuconv}
is $H$ and $\gamma'=l(\cos\phi,\sin\phi)$.
Then
$x,y$
satisfy the differential equation:
\begin{equation}\label{eq:mean}
2H(t)y(t)l(t) - 
l(t)\cos\phi(t)+y(t)\phi'(t)
 = 0.
\end{equation}
Following Kenmotsu \cite{k1}, we solve this equation
together with 
\begin{equation}\label{eq:tangent}
(x'(t),y'(t))=l(t)(\cos\phi(t),\sin\phi(t)).
\end{equation}
We set 
$z(t)=y(t)\sin\phi(t)+\sqrt{-1}y(t)\cos\phi(t)$.
Then \eqref{eq:mean} can be modified into
$$
z'(t)-2\sqrt{-1}H(t)z(t)l(t)-l(t)=0,
$$
and the general solution of this equation is
\begin{align*}
&z(t)=
(F(t)-c_1)\sin\eta(t)+(G(t)-c_2)\cos\eta(t)\\
&\hspace{20mm}+\sqrt{-1}\big(
(G(t)-c_2)\sin\eta(t)-(F(t)-c_1)\cos\eta(t)
\big),
\end{align*}
where $c_1,c_2\in\R$, and
$$
F(t)=\int_0^t l(u)\sin\eta(u)\, du,\quad
G(t)=\int_0^t l(u)\cos\eta(u)\, du,\quad
\eta(u)=\int_0^u 2l(v)H(v)\,dv.
$$
By $y(t)^2=|z(t)|^2$ and 
$x'(t)=l(t)\cos\phi(t)=l(t)(z(t)-\bar{z}(t))/(2\sqrt{-1}y(t))$,
we have
\begin{align}
\label{eq:soly}
y(t)&=
((F(t)-c_1)^2+(G(t)-c_2)^2)^{1/2},\\
\label{eq:solx}
x'(t)&=
\dfrac{F'(t)(G(t)-c_2)-G'(t)(F(t)-c_1)}
{((F(t)-c_1)^2+(G(t)-c_2)^2)^{1/2}}
=\dfrac{F'(t)(G(t)-c_2)-G'(t)(F(t)-c_1)}
{y(t)}.
\end{align}
We take the initial values
$c_1,c_2$ satisfying that 
$(F(t)-c_1)^2+(G(t)-c_2)^2>0$
on the considering domain.

It should be mentioned that on the set of regular points,
there is Kenmotsu's result \cite{k1},
and singularities can be considered by taking 
the limits of regular parts. 
However, we will see the class of 
singularities of $\gamma$ in Section \ref{sec:sing},
which cannot be investigated just looking at
limits of regular points.
Furthermore, we believe that
the formula \eqref{eq:solx}, \eqref{eq:soly},  
which is able to pass though the singularities,
can extend the treatment of 
singular surfaces of revolution.
We remark that there is a representation formula
\cite[Theorem 4]{k3}
for surfaces which have prescribed $H$ and 
the unit normal vector.
\section{Singularities of profile curves}\label{sec:sing}
In this section, we study conditions for singularities
of profile curves and revolution surfaces.
A singular point $p$ of a map $\gamma$ is 
called a {\it ordinary cusp\/} or {\it $3/2$-cusp\/}
if the map-germ $\gamma$ at $p$ is $\mathcal{A}$-equivalent to
$t\mapsto(t^2,t^3)$ at $0$ (Two map-germs
$f_1,f_2:(\R^m,0)\to(\R^n,0)$ are $\mathcal{A}$-{\it
equivalent}\/ if there exist diffeomorphisms
$S:(\R^m,0)\to(\R^m,0)$ and $T:(\R^n,0)\to(\R^n,0)$ such
that $ f_2\circ S=T\circ f_1 $.).
Similarly, a singular point $p$ of a map $\gamma$ is 
called a {\it $j/i$-cusp\/}
if the map-germ $\gamma$ at $p$ is $\mathcal{A}$-equivalent to
$t\mapsto(t^i,t^j)$ at $0$,
where $(i,j)=(2,5),(3,4),(3,5)$.
It is known that the singularity of $(\R,0)\to(\R^2,0)$ 
which are determined by its $5$-jet 
with respect to $\A$-equivalence are only these cusps.
Criteria for these singularities are known.
See \cite{bg} for example.
\begin{fact}\label{fact:cri}
A map-germ\/ $\alpha:(\R,p)\to \R^2$
satisfying\/ $\alpha'(p)=0$ is 
\begin{itemize}
\item a\/ $3/2$-cusp if and only if\/ 
$\det(\alpha'',\alpha''')\ne0$ holds at\/ $p$,
\item a\/ $5/2$-cusp if and only if\/ 
$\alpha''\ne0$,
$\alpha'''=k\alpha''$ and\/
$\det(\alpha'',3\alpha^{(5)}-10k\alpha^{(4)})\ne0$ hold at\/ $p$,
where $(~)^{(i)}=d^i/dt^i$,
\item a\/ $4/3$-cusp if and only if\/
$\alpha''=0$ and\/
$\det(\alpha''',\alpha^{(4)})\ne0$ hold at\/ $p$,
\item a\/ $5/3$-cusp if and only if\/
$\alpha''=0$,
$\det(\alpha''',\alpha^{(4)})=0$ and\/
$\det(\alpha''',\alpha^{(5)})\ne0$ hold at\/ $p$.
\end{itemize}
\end{fact}
A map-germ $\gamma$ at $p$ is called {\it frontal\/}
if there exists a map $n:(\R,p)\to(\R^2,0)$ 
satisfying $|n|=1$ and $\gamma'\cdot n=0$ for any $t$.
A frontal is a {\it front\/} if the pair $(\gamma,n)$ is 
an immersion.
If $\gamma$ at $p$ is a $3/2$-cusp or a $4/3$-cusp then it
is a front, and 
if $\gamma$ at $p$ is a $5/2$-cusp or a $5/3$-cusp then it
is a frontal but not a front.
By definition, $\gamma'(p)=0$ if and only if $l(p)=0$.
We have the following:
\begin{proposition}\label{prop:frontal}
The curve\/ $\gamma=(x,y)$ 
given by\/ \eqref{eq:soly}, \eqref{eq:solx}
is a frontal at any point.
Moreover, if\/ $l(p)=0$, then\/
$\gamma$ at\/ $p$ is a front if and only if\/
$\eta'(p)\ne0$.
\end{proposition}
\begin{proof}
Since
$
y'
=
(F'(F-c_1)+G'(G-c_2))y^{-1},
$
we have
\begin{equation}\label{eq:gammap}
\gamma'
=
\dfrac{F'}{y}
\pmt{
G-c_2\\
F-c_1}
+
\dfrac{G'}{y}
\pmt{-(F-c_1)\\
G-c_2
}
=
\dfrac{l\cos\eta}{y}U
+
\dfrac{l\sin\eta}{y}
\pmt{0&1\\-1&0}
U
=
\dfrac{l}{y}R_{-\eta} U,
\end{equation}
where $$
\gamma'=\pmt{
x'\\
y'},\quad
U=\pmt{-(F-c_1)\\G-c_2}\quad\text{and}\quad 
R_{-\eta}=\pmt{\cos(-\eta)&-\sin(-\eta)\\ \sin(-\eta)&\cos(-\eta)}.$$
We set
$n=R_{-\eta+\pi/2}U/|U|$.
Then $|n|=1$ and $n$ is perpendicular to $\gamma'$.
Thus $\gamma$ is a frontal.
Let us assume $l(p)=0$. Then $\gamma$ at $p$ is a front
if and only if $n'(p)\ne0$.
This is equivalent to saying that $R_{-\eta+\pi/2}U$ and
$(R_{-\eta+\pi/2}U)'$ are linearly independent.
Since $l(p)=0$, 
it holds that $(R_{-\eta+\pi/2}U)'(p)=R_{-\eta+\pi/2}'(p)U(p)$,
and $R_{-\eta+\pi/2}'=-\eta'R_{-\eta+\pi}$,
we see that
$n'(p)\ne0$ is equivalent to $\eta'(p)\ne0$.
This proves the assertion.
\end{proof}
Moreover, we have the following:
\begin{proposition}\label{prop:singcond}
Let\/ $\gamma=(x,y)$ is given by\/ \eqref{eq:soly} 
and\/ \eqref{eq:solx}.
We assume that\/ $l(p)=0$,
then\/ $\gamma$ at\/ $p$ is
\begin{enumerate}
\item\label{itm:23}
 a\/ $3/2$-cusp if and only if\/ $l'\eta'\ne0$ holds at\/ $p$,
\item\label{itm:25}
 a\/ $5/2$-cusp if and only if\/ $l'\ne0$, 
$\eta'=0$ and\/
$l''\eta''-l'\eta'''\ne0$ hold at\/ $p$,
\item\label{itm:34}
 a\/ $4/3$-cusp if and only if\/ $l'=0$ and\/
$\eta'l''\ne0$ hold at\/ $p$,
\item\label{itm:35}
 a\/ $5/3$-cusp if and only if\/ 
$l'=\eta'=0$ and\/
$\eta''l'' \ne0$ hold at\/ $p$.
\end{enumerate}
\end{proposition}
\begin{proof}
By \eqref{eq:gammap},
we have
\begin{equation}\label{eq:gammapp}
\gamma''
=
l'\left(y^{-1}\right)R_{-\eta} U
+
l\left(y^{-1}\right)'R_{-\eta} U
+
l\left(y^{-1}\right)R_{-\eta}' U
+
l\left(y^{-1}\right)R_{-\eta} U'
\end{equation}
and since $l(p)=0$, so $y'(p)=0$ and $U'(p)=0$ hold.
Then we have
$\gamma''(p)=l'(p)y(p)^{-1}R_{-\eta}(p) U(p)$.
Thus $\gamma''(p)\ne0$ if and only if $l'(p)\ne0$.
We assume that $l'(p)\ne0$.
Then by \eqref{eq:gammapp},
\begin{align}
\label{eq:gammappp}
\gamma'''
=&
l''\left(y^{-1}\right)R_{-\eta} U
+
l\left(y^{-1}\right)''R_{-\eta} U
+
l\left(y^{-1}\right)R_{-\eta}'' U
+
l\left(y^{-1}\right)R_{-\eta} U''\\
&+
2l'\left(y^{-1}\right)'R_{-\eta}U
+
2l'\left(y^{-1}\right)R_{-\eta}'U
+
2l'\left(y^{-1}\right)R_{-\eta}U'\nonumber\\
&+
2l\left(y^{-1}\right)'R_{-\eta}'U
+
2l\left(y^{-1}\right)'R_{-\eta}U'
+
2l\left(y^{-1}\right)R_{-\eta}'U',\nonumber
\end{align}
and since $l(p)=y'(p)=0$, $U'(p)=0$, and
$R_{-\eta}'=(-\eta')R_{-\eta+\pi/2}$,
$$\gamma'''=
\Big(l''R_{-\eta} U
+
2l'R_{-\eta}'U\Big)y^{-1}
=
\Big(l''R_{-\eta} U
+
2l'(-\eta')R_{-\eta+\pi/2}U\Big)y^{-1}
$$
holds at $p$.
Hence $\det(\gamma'',\gamma''')(p)\ne0$ if and only if
$\eta'(p)\ne0$, and this proves \ref{itm:23}.
We assume $\eta'(p)=0$.
Then we see $k$ in \ref{itm:25} is $l''(p)/l'(p)$.
Now we calculate $\det(\gamma'',3\gamma^{(5)}-10k\gamma^{(4)})(p)$.
Differentiating \eqref{eq:gammapp} with noticing
$l(p)=y'(p)=\eta'(p)=0$ and $U'(p)=0$,
we have
$$
3\gamma^{(5)}-10k\gamma^{(4)}
=\Big(
- l'' R_{-\eta}'' U 
+ l'  R_{-\eta}''' U
- l'' R_{-\eta} U'' 
+ l'  R_{-\eta} U''' \Big)y^{-1}+*R_{-\eta}U
$$
at $p$, where $*$ stands for a real number.
Then
we see that
$\det(\gamma'',3\gamma^{(5)}-10k\gamma^{(4)})(p)=
12 l'(p) \big(l''(p) \eta''(p)-l'(p) \eta'''(p)\big)$.
This proves the assertion \ref{itm:25}.

Next we assume $\gamma''(p)=0$, namely, $l'(p)=0$.
Then by \eqref{eq:gammappp}, 
$$
\gamma'''=l''(y^{-1})R_{-\eta}U,\quad
\gamma^{(4)}=
3 l''(y^{-1}) R_{-\eta}' U
+
*R_{-\eta}U.
$$
Since $R_{-\eta}'=-\eta R_{-\eta+\pi/2}$,
this proves \ref{itm:34}.
We assume that $\eta'(p)=0$.
Then differentiating \eqref{eq:gammappp} twice, 
we have
$$
\gamma^{(5)}
=
6 l''(R_{-\eta}'' U+R_{-\eta}  U'') y^{-1} 
+*R_{-\eta}U.
$$
at $p$.
Since $U''=l'\,\trans{(\sin\eta,\cos\eta)}+l\,\trans{(\sin\eta,\cos\eta)'}=0$
at $p$,
we have \ref{itm:35},
where $\trans{(~)}$ stands for the matrix transportation.
\end{proof}
\begin{example}\label{ex:23}
Let us set $H=1/t$ and $l=t$ with $c_1=c_2=1/10$.
Then by Proposition \ref{prop:singcond},
$\gamma$ at $t=0$ is $3/2$-cusp.
The profile curve can be drawn 
as in Figure \ref{fig:23}.
\begin{figure}[ht]
\centering
\begin{tabular}{ccc}
\includegraphics[width=.4\linewidth]
{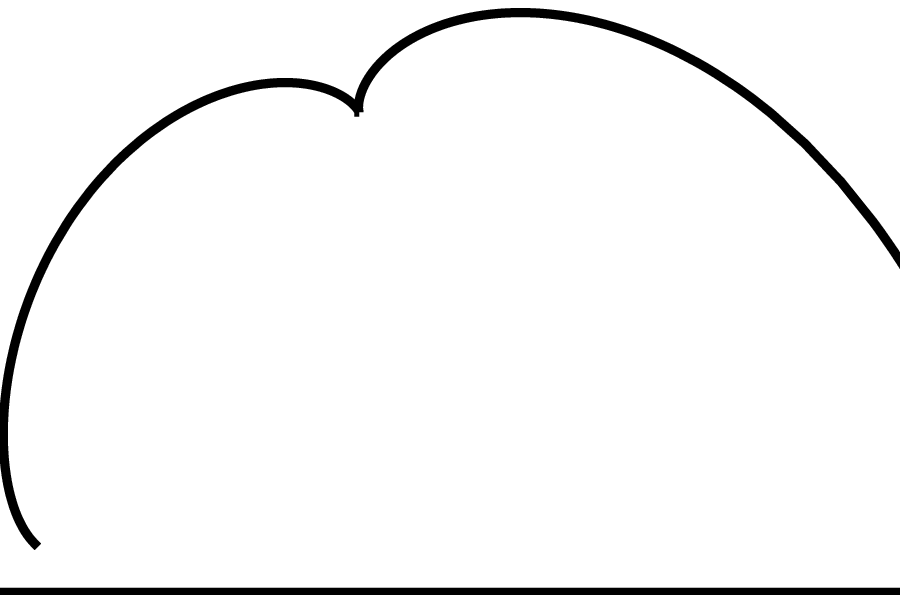}&
\hspace*{10mm}&
\includegraphics[width=.4\linewidth]
{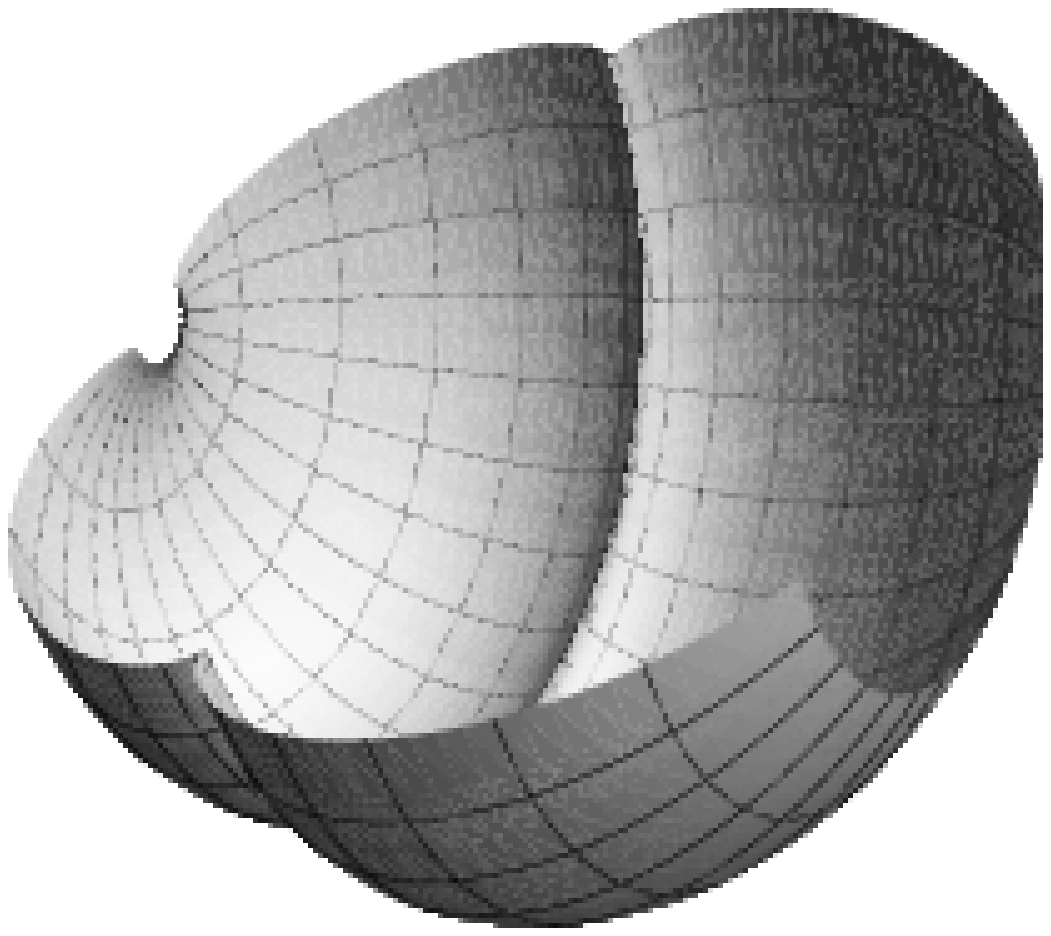}
\end{tabular}
\caption{The profile curve and the revolution surface of Example
\ref{ex:23}.
The horizontal line stands for the $x$-axis.} 
\label{fig:23}
\end{figure}
\end{example}
\begin{example}\label{ex:253435}
Let us set $H=1+t$ and $l=t$ with $c_1=c_2=1/10$.
Then by Proposition \ref{prop:singcond},
$\gamma$ at $t=0$ is $5/2$-cusp.
Let us set $H=1/t^2$ and $l=t^2$ with $c_1=c_2=1/10$.
Then by Proposition \ref{prop:singcond},
$\gamma$ at $t=0$ is $4/3$-cusp.
Let us set $H=1/t$ and $l=t^2$ with $c_1=c_2=1/10$.
Then by Proposition \ref{prop:singcond},
$\gamma$ at $t=0$ is $5/3$-cusp.
The profile curves can be drawn 
as in Figure \ref{fig:253435}.
The singular points are indicated by the arrows.
\begin{figure}[ht]
\centering
\begin{tabular}{ccccc}
\includegraphics[width=.16\linewidth]
{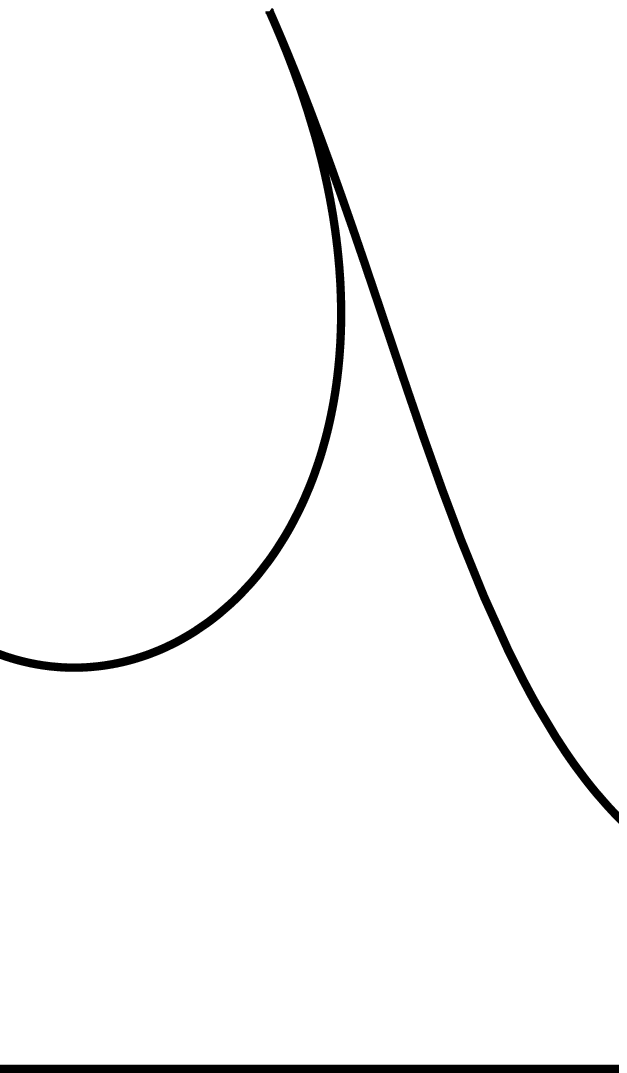}&
\hspace{5mm}&
\includegraphics[width=.25\linewidth]
{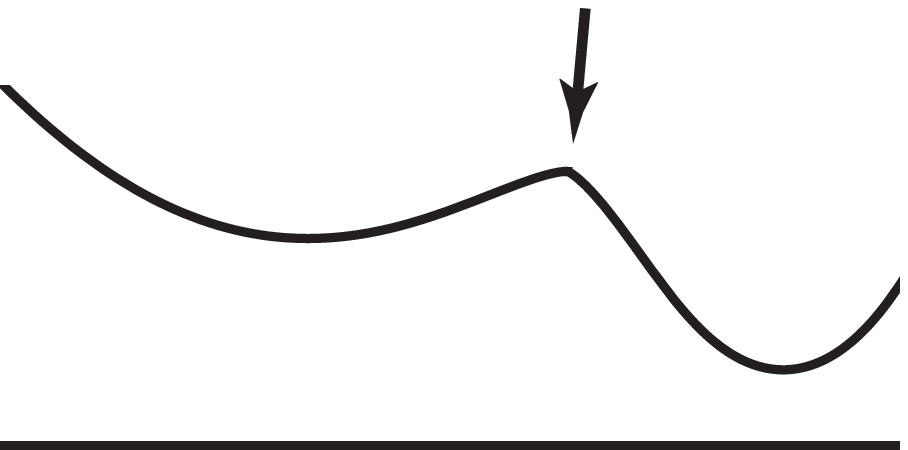}&
\hspace{5mm}&
\includegraphics[width=.18\linewidth]
{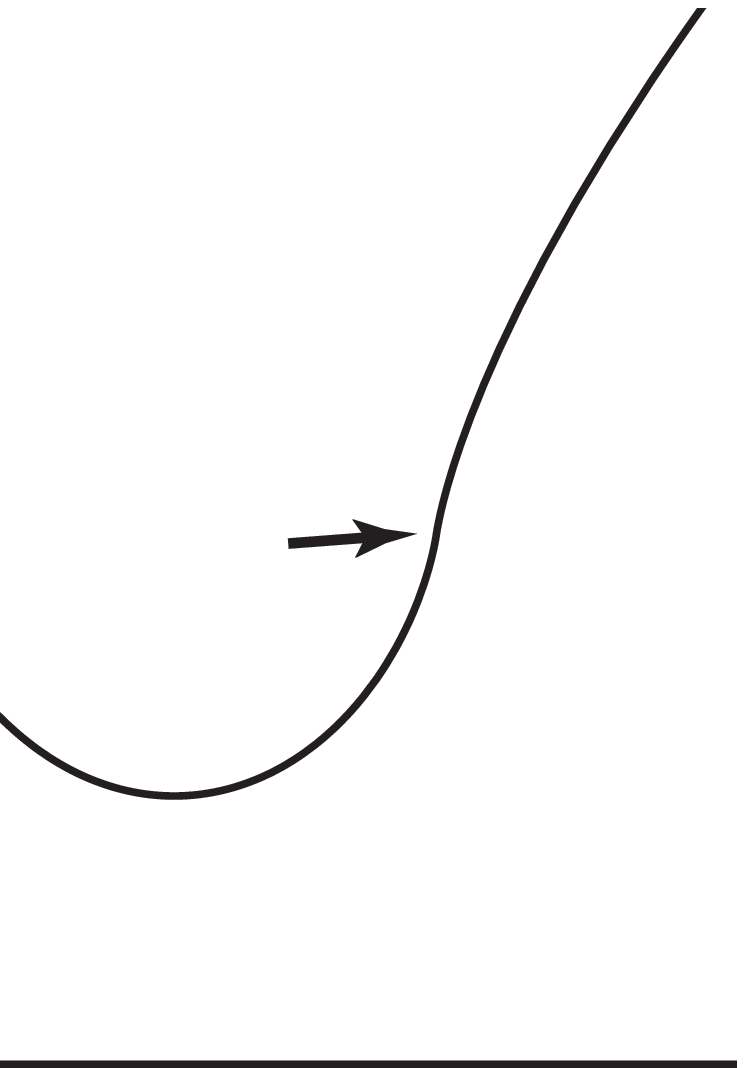}
\end{tabular}
\caption{The profile curves of Example
\ref{ex:253435}.
The horizontal lines stand for the $x$-axis.} 
\label{fig:253435}
\end{figure}
\end{example}

By Proposition \ref{prop:singcond},
we can study the singularities of the revolution surface.
A singular point $q$ of a map $f:(\R^2,q)\to(\R^3,0)$ is 
called a {\it $j/i$-cuspidal edge\/}
if $f$ at $q$ is $\mathcal{A}$-equivalent to
$(u,v)\mapsto(u^i,u^j,v)$ at $0$.
Since the map-germ
$(x,y,z)\mapsto(x,y\cos z,y\sin z)$
is a diffeomorphism if $y\ne0$,
the map-germ $s$ in \eqref{eq:revol} at $(p,\theta)$
is a $j/i$-cuspidal edge if and only if
the profile curve $\gamma=(x,y)$ at $p$ is 
a $j/i$-cusp.

\section{Periodicity}
In this section, we study the condition for
periodicity of 
surfaces when $H$ and $l$ are periodic,
where the condition for regular case is obtained 
by Kenmotsu \cite{k2}.
We define the profile curve $(x,y)$ of 
the surface of revolution given by
\eqref{eq:revol} being {\it periodic\/}
with the period $L$
if there exists $T>0$ such that
$x(s+L)=x(s)+T$ and $y(s+L)=y(s)$.
Then we have the following theorem.
\begin{theorem}\label{thm:peri}
Let\/ $H:\R\setminus P\to\R$
and\/ $l:\R\to\R$ 
be periodic\/ $C^\infty$ functions
of the same period\/ $L$, where\/ $P$ is a discrete set
satisfying that $Hl$ can be extended to a $C^\infty$ function
on $\R$ and $P=l^{-1}(0)$.
Then the solution\/ $(x,y)$ in\/
\eqref{eq:solx}, \eqref{eq:soly} of\/ \eqref{eq:mean}
with\/ \eqref{eq:tangent} is periodic
if and only if\/ $1-\cos\eta(L)\ne0$ and
\begin{equation}\label{eq:period050}
\cos\!\left(\phi(0)+\dfrac{\eta(L)}{2}\right)
\int_0^L l(u)\sin\eta(u)\,du
=
\sin\!\left(\phi(0)+\dfrac{\eta(L)}{2}\right)
\int_0^L l(u)\cos\eta(u)\,du,
\end{equation}
or\/ $1-\cos\eta(L)=0$ and
\begin{equation}\label{eq:period070}
\int_0^L l(u)\sin\eta(u)\,du
=
\int_0^L l(u)\cos\eta(u)\,du
=0,
\end{equation}
where $(x'(0),y'(0))=l(0)(\cos\phi(0),\sin\phi(0))$.
\end{theorem}
Kenmotsu gave the condition for 
the case of the profile curve is regular \cite[Theorem 1]{k2}.
If the profile curve is regular, the above condition 
is the same as Kenmotsu's condition.
In fact, for regular case, since one can take $t=0$ giving the minimum
of $y$, we can assume that $\phi(0)=0$.
However, in our case, the profile curve may have singularities,
the existence of $t_0$ such that $\phi(t_0)=0$ fails in general.
One can show Theorem \ref{thm:peri} by the similar method to 
Kenmotsu \cite[Theorem 1]{k2},
we give a proof here for the completion.
\begin{proof}[Proof of Theorem\/ {\rm \ref{thm:peri}}]
Let us assume 
$l(0)\ne0$ and $(x'(0),y'(0))=l(0)(\cos\phi(0),\sin\phi(0))$.
By \eqref{eq:gammap} together with $y(0)=y(L)$, $y'(0)=y'(L)$ 
$x'(0)=x'(L)$ and $l(0)=l(L)$
yield that
\begin{align}
\label{eq:period100}
-c_2
&=
\sin \eta(L) (F(L) - c_1) + \cos \eta(L) (G(L) - c_2),\\
\label{eq:period200}
c_1
&=
\sin\eta(L)(G(L) - c_2) - \cos\eta(L) (F(L) - c_1).
\end{align}
If $1-\cos\eta(L)\ne0$ then,
\eqref{eq:period100} and \eqref{eq:period200} is equivalent to
\begin{align}
\label{eq:period300}
c_1&=\dfrac{F(L)-F(L) \cos\eta(L)+G(L) \sin\eta(L)}{2(1-\cos\eta(L))},\\
\label{eq:period400}
c_2&=\dfrac{G(L)-G(L) \cos\eta(L)-F(L) \sin\eta(L)}{2(1-\cos\eta(L))}.
\end{align}
On the other hand, by \eqref{eq:soly}, \eqref{eq:solx},
$(\cos\phi(0),\sin\phi(0))$ is parallel to $(c_1,-c_2)$,
$$
\det\pmt{
\cos\phi(0)&F(L)-F(L) \cos\eta(L)+G(L) \sin\eta(L)\\
\sin\phi(0)&-(G(L)-G(L) \cos\eta(L)-F(L) \sin\eta(L))}=0.
$$
This is equivalent to \eqref{eq:period050}.
If $1-\cos\eta(L)=0$,
\eqref{eq:period100} and \eqref{eq:period200} are equivalent to
$F(L)=G(L)=0$, and this implies \eqref{eq:period070}.

Conversely, we assume that for periodic functions
$H$ and $l$ with period $L$ satisfy
the condition $1-\cos\eta(L)\ne0$ and \eqref{eq:period050},
or $1-\cos\eta(L)=0$ and \eqref{eq:period070}.
By definition of $\eta$, we have $\eta(u+L)=\eta(u)+\eta(L)$.
Then by definitions of $F,G$,
we have
\begin{align*}
F(t+L)=F(L)+\sin\eta(L)G(t)+\cos\eta(L)F(t),\\
G(t+L)=G(L)+\cos\eta(L)G(t)-\sin\eta(L)F(t).
\end{align*}
If $1-\cos\eta(L)\ne0$, a direct calculation shows that
$y$ given by \eqref{eq:soly}
with \eqref{eq:period300}, \eqref{eq:period400}
satisfies $y(t+L)=y(t)$,
and also we see that
$x'$ given by \eqref{eq:solx}
with \eqref{eq:period300}, \eqref{eq:period400}
satisfies $x'(t+L)=x'(t)$.
If $1-\cos\eta(L)=0$, then $\eta(L)=0$, and
we have $F(L)=G(L)=0$.
This shows the desired periodicities of $x$ and $y$.
\end{proof}

\begin{example}\label{ex:period1}
Let us set $H=1/\sin t$ and $l=\sin t$ with $c_1=1$, 
$c_2=3/4$.
This satisfies the condition in Theorem \ref{thm:peri},
and the profile curve is periodic.
The profile curve can be drawn 
as in Figure \ref{fig:periodic}.
Each singularity is $3/2$-cusp by Proposition
\ref{prop:singcond}.
\begin{figure}[ht]
\centering
\begin{tabular}{ccc}
\includegraphics[width=.4\linewidth]
{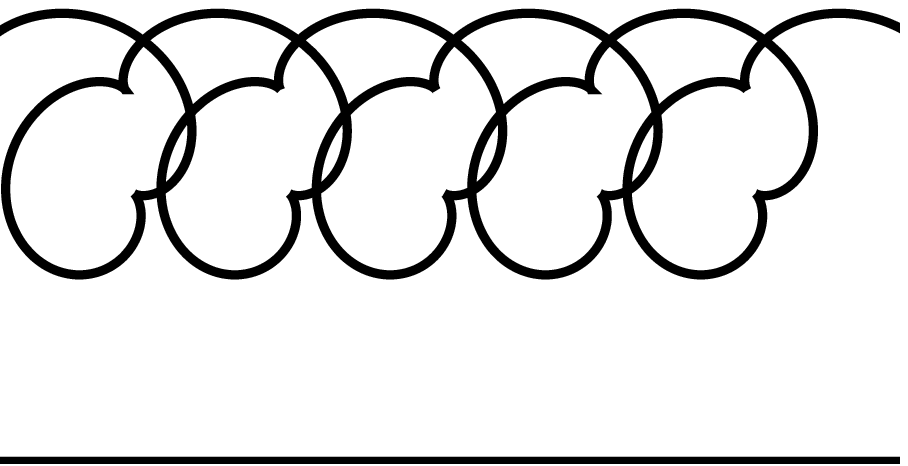} &\hspace*{15mm}&
\includegraphics[width=.3\linewidth]
{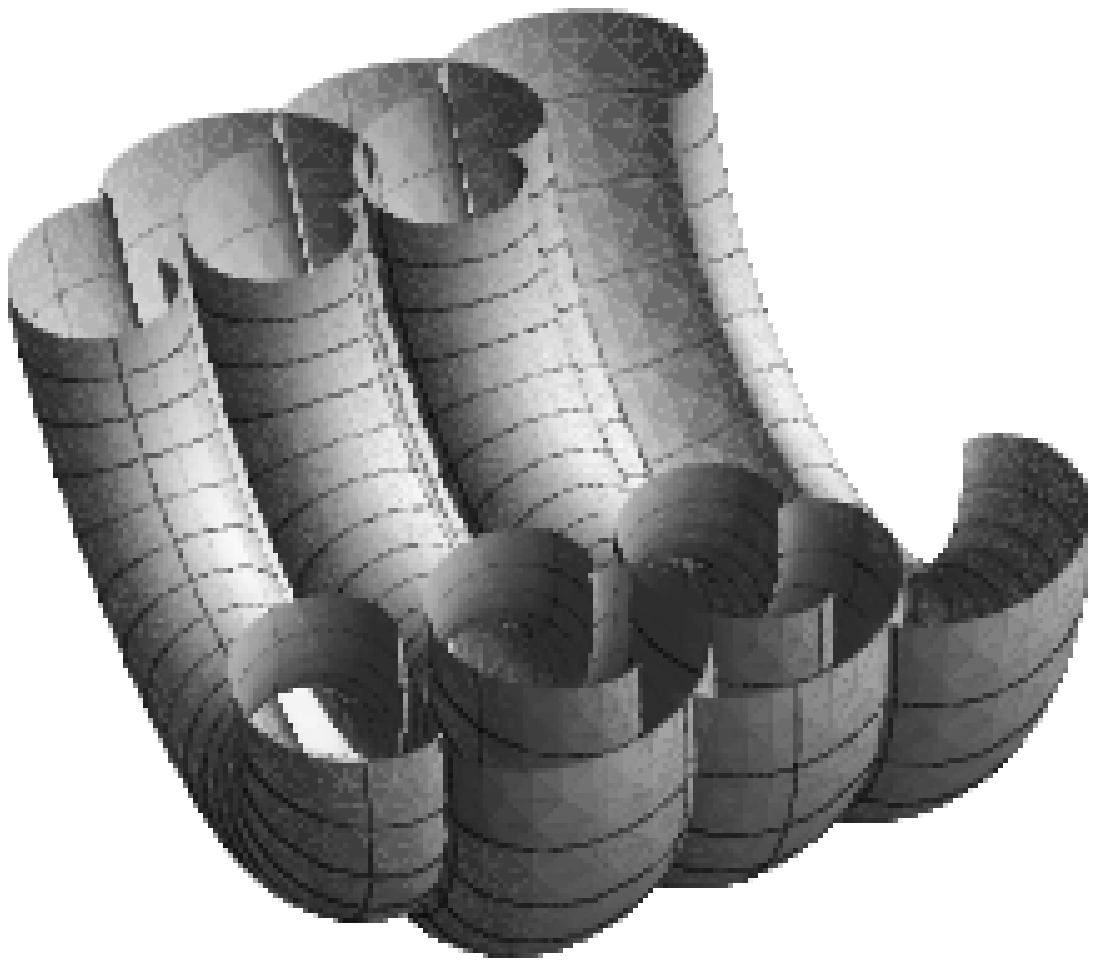}
\end{tabular}
\caption{Profile curve and the revolution surface of Example
\ref{ex:period1}.
The horizontal line stands for the $x$-axis.} 
\label{fig:periodic}
\end{figure}
\end{example}
\begin{example}\label{ex:period2tan}
Let us set $H=\tan t$ and $l=\cos t$
with $c_1=c_2=1/10$.
A numerical computation shows that
$H$ and $l$ do not satisfy the condition in 
Theorem \ref{thm:peri},
the profile curve is not periodic
as we can see in Figure \ref{fig:period2tan}.
Each singularity is $3/2$-cusp by Proposition
\ref{prop:singcond}.
\begin{figure}[ht]
\centering
\begin{tabular}{ccc}
\includegraphics[width=.05\linewidth]
{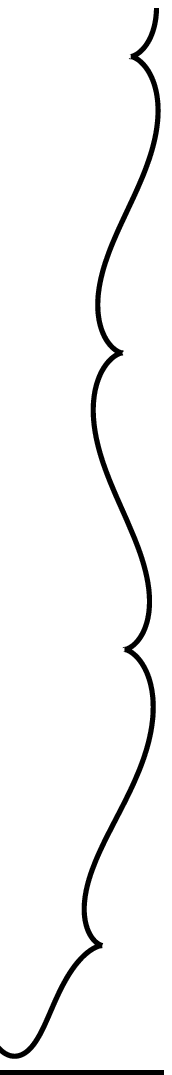} &\hspace*{10mm}&
\includegraphics[width=.2\linewidth]
{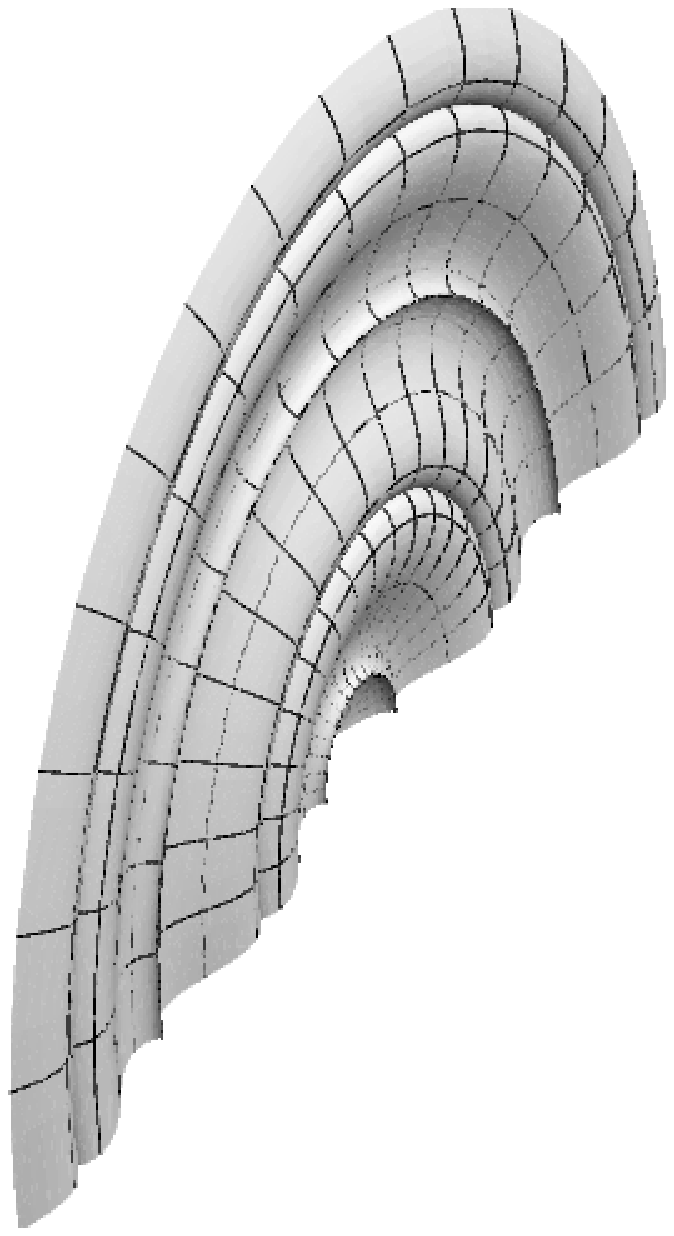}
\end{tabular}
\caption{Profile curve and the revolution surface of Example
\ref{ex:period2tan}.
The horizontal line stands for the $x$-axis.} 
\label{fig:period2tan}
\end{figure}
\end{example}
\begin{example}\label{ex:period2}
Let us set $H=1/\sin^2 t$ and $l=\sin^2 t$
with $c_1=c_2=1/10$.
This does not satisfy the condition in Theorem \ref{thm:peri},
the profile curve is not periodic
as we can see in Figure \ref{fig:periodic2}.
Each singularity is $4/3$-cusp by Proposition
\ref{prop:singcond}, and they are indicated by the arrows.

\begin{figure}[ht]
\centering
\begin{tabular}{ccc}
\includegraphics[width=.12\linewidth]
{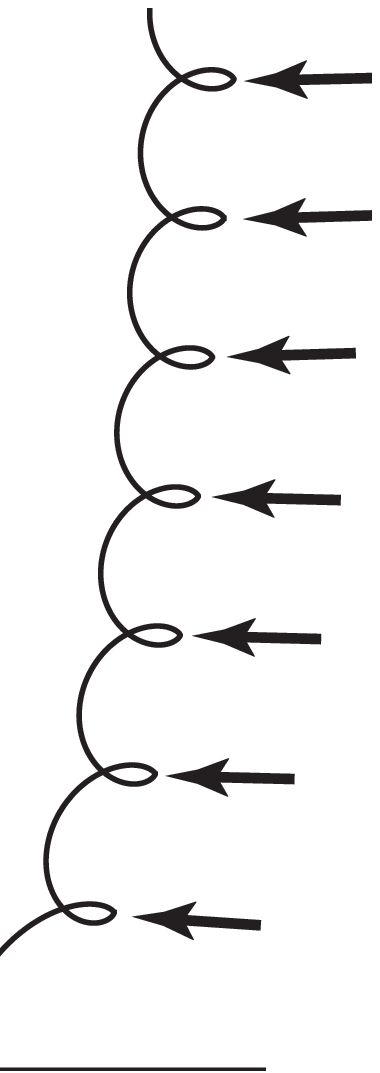} &\hspace*{10mm}&
\includegraphics[width=.55\linewidth]
{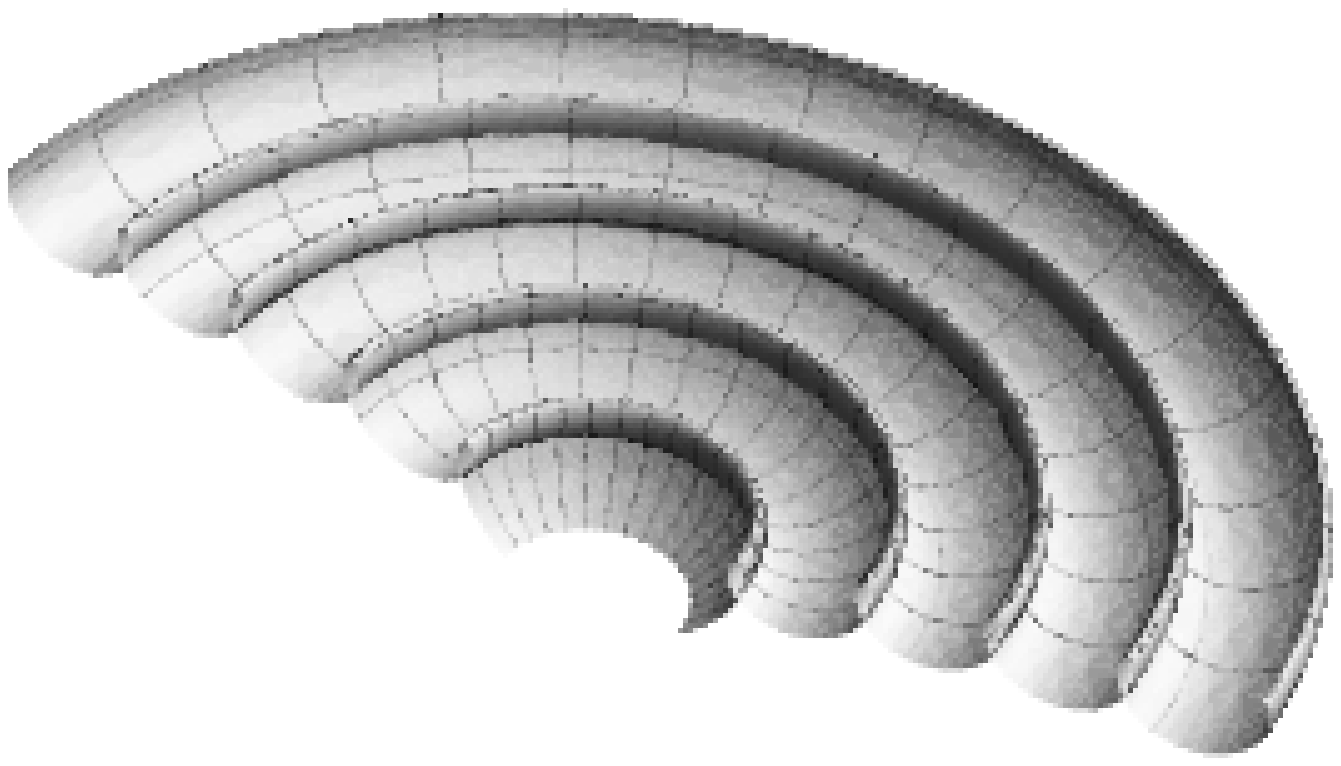}
\end{tabular}
\caption{Profile curve and the revolution surface of Example
\ref{ex:period2}.
The horizontal line stands for the $x$-axis.} 
\label{fig:periodic2}
\end{figure}
\end{example}

The authors would like to thank Kenichi Ito for helpful
advices, and Yoshihito Kohsaka for encouragement.



\medskip
{\footnotesize
\begin{flushright}
\begin{tabular}{ll}
\begin{tabular}{l}
(Martins and Santos)\\
Departamento de Matem{\'a}tica,\\
Instituto de Bioci\^{e}ncias, \\
Letras e Ci\^{e}ncias Exatas,\\
 UNESP - Universidade Estadual Paulista,\\
  C\^{a}mpus de S\~{a}o Jos\'{e} do Rio Preto, \\
SP, Brazil\\
  E-mail: {\tt lmartinsO\!\!\!aibilce.unesp.br}\\
  E-mail: {\tt samuelp.santosO\!\!\!ahotmail.com}
\end{tabular}
&
\begin{tabular}{l}
(Saji and Teramoto)\\
Department of Mathematics,\\
Graduate School of Science, \\
Kobe University, \\
Rokkodai 1-1, Nada, Kobe \\
657-8501, Japan\\
  E-mail: {\tt sajiO\!\!\!amath.kobe-u.ac.jp}\\
  E-mail: {\tt teramotoO\!\!\!amath.kobe-u.ac.jp}\\
\phantom{a}
\end{tabular}
\end{tabular}
\end{flushright}}
\end{document}